\documentclass[11pt]{article}

\usepackage{amstext} 
\usepackage{latexsym}
\usepackage{euscript}

\usepackage{amsmath}
\usepackage{amsfonts}
\usepackage{amssymb}

\newcommand{\simi}{\sim\!\!}
\newtheorem{lem}{\bf Lemma}[section]
\newtheorem{teo}{\bf Theorem}[section]
\newtheorem{cor}{\bf Corollary}[section]
\newtheorem{defi}{\bf Definition}[section]
\newtheorem{exa}{\bf Example}[section]
\newtheorem{rem}{\bf Remark}[section]

\def\nd{{\scriptstyle \Gamma}\,}

\newenvironment{proof}{\noindent\bf Proof.\rm}{\hfill $\mbox{\boldmath{$\Box$}}$}

\title{Several characterizations of the 4--valued modal algebras}
\author{Aldo V. Figallo and Paolo Landini}
\date{}

\begin{document}

\maketitle

\begin{abstract}
A. Monteiro, in 1978, defined the algebras he named tetravalent modal algebras,  will be called  {\em $4-$valued modal algebras} in this work. These algebras constitute a generalization of the $3-$valued Lukasiewicz algebras defined by Moisil.

The theory of the $4-$valued modal algebras has been widely developed by I. Loureiro in \cite{IL1, IL2, IL3, IL4, IL5, IL6, IL7} and by A. V. Figallo in \cite{FI3, FI4, AF.PL,AF.AZ}. 

J. Font and M. Rius indicated, in the introduction to the important work \cite{JF.MR2}, a brief but detailed review on the $4-$valued modal algebras.

In this work varied characterizations are presented that show the ``closeness'' this variety of algebras has with other well--known algebras related to the algebraic counterparts of certain logics.

\end{abstract}

\section{Introduction}

In 1940 G. C. Moisil \cite{MOI1} introduced the notion of three--valued $\L$%
ukasiewicz algebra. In 1963, A.Monteiro \cite{AMO1} characterized these
algebras as algebras $\langle A, \wedge, \vee, \sim, \triangledown,1\rangle$ of type 
$(2,2,1,1,0)$ which verify the following identities:

\begin{itemize}
\item[\rm (A1)]  $x\vee 1=1$,

\item[\rm (A2)]  $x\wedge (x\vee y)=x$,

\item[\rm (A3)]  $x\wedge (y\vee z)=(z\wedge x)\vee (y\wedge x)$,

\item[\rm (A4)]  $\sim \simi  x=x$,

\item[\rm (A5)]  $\simi  (x\vee y)=\simi  x\wedge \simi  y$,

\item[\rm (A6)]  $\simi  x\vee \triangledown x=1$,

\item[\rm (A7)]  $x\wedge \simi  x=\simi  x\wedge \triangledown x$,

\item[\rm (A8)]  $\triangledown (x\wedge y)=\triangledown x\wedge \triangledown y$.
\end{itemize}

L.Monteiro \cite{LMO1} proved that A1 follows from A2, $\cdots$, A8, and
that A2, $\cdots$, A8, are independent.

From A2, $\cdots$, A5 it follows that $\langle A, \wedge, \vee, \sim,1\rangle$ is a De
Morgan algebra with last element $1$ and first element $0 = \simi  1$.

In 1969 J. Varlet \cite{VAR1} characterized three--valued {\L}ukasiewicz algebras by
means of other operations. Let $\langle A, \wedge, \vee, *, +,0,1\rangle$ be an algebra
of type $(2,2,1,1,0,0)$ where $\langle A, \wedge, \vee,0,1\rangle$ is a bound distributive lattice with least element $0$, greatest element $1$ and the following properties are satisfied:

\begin{itemize}
\item[\rm (V1)]  $x\wedge x^{*}=0$,

\item[\rm (V2)]  $(x\wedge y)^{*}=x^{*}\wedge y^{*}$,

\item[\rm (V3)]  $0^{*}=1$,

\item[\rm (V4)]  $x\vee x^{+}=1$,

\item[\rm (V5)]  $(x\vee y)^{+}=x^{+}\wedge y^{+}$,

\item[\rm (V6)]  $1^{+}=0$,

\item[\rm (V7)]  if $x^{*}=y^{*}$ and $x^{+}=y^{+}$ then $x=y$.
\end{itemize}

About these algebras he proved that it is posible to define, in the sense of
\cite{AMO1,LMO1} a structure of three-valued {\L}ukasiewicz algebra by taking $\simi  x =(x \vee x^*) \wedge x^+$ and  $\triangledown x = x^{**}$.

Furthermore it holds $x^*=\simi  \triangledown x$ and $x^+=\triangledown \simi  x$. Therefore three--valued {\L}ukasiewicz are double Stone lattices which satisfy the determination
principle V7. Moreover V7 may be replaced by the identity 
$$(x \wedge x^+)\wedge (y \vee y^*)= x \wedge x^+.$$

Later, in 1978, A. Monteiro \cite{AMO1} considered the $4$--valued modal algebras \linebreak $\langle A, \wedge, \vee, \sim, \triangledown,1\rangle$ of type $(2,2,1,1,0)$ which satisfy A2, $\cdots$, A7 as an abstraction of three--valued {\L}ukasiewicz algebras.

In this paper we give several characterizations of the 4--valued modal algebras. In the first one we consider the operations $\wedge, \vee, \neg,\nd,0,1$ where $\neg x = \simi  \triangledown x$, $\nd x = \triangledown \simi  x$ are called strong and weak negation respectively.

\section{A characterization of the 4--valued modal \\
algebras}

\begin{teo}\label{2.1} 
Let $\langle A,\wedge ,\vee ,\neg ,\nd ,0,1\rangle$ be an algebra of type
$(2,2,1,1,0,0)$ where  $\langle A,\wedge ,\vee ,0,1\rangle$ is a bounded distributive lattice with least element $0$, \linebreak greatest element $1$ and the operators $\triangledown ,\sim $ are defined on $A$ by means of the formulas:

\begin{itemize}
\item[\rm (D1)]  $\triangledown x=\neg \neg x$,

\item[\rm(D2)]  $\simi  x=(x\vee \neg x)\wedge \nd x$.
\end{itemize}
\noindent Then $\langle A,\wedge ,\vee ,\sim ,\triangledown ,1\rangle$ is a $4$--valued modal algebra if and only if it satisfies the following properties:

\begin{itemize}
\item[\rm (B1)]  $x\wedge \neg x=0$,

\item[\rm (B2)]  $x\vee \nd x=1$,

\item[\rm (B3)]  $\neg x\wedge \nd \neg x=0$,

\item[\rm (B4)]  $\nd x\vee \neg \nd x=1$,

\item[\rm (B5)]  $\nd (x\wedge y)=\nd x\vee \nd y$,

\item[\rm (B6)]  $\neg (x\vee y)=\neg x\wedge \neg y$,

\item[\rm (B7)]  $\neg (x\wedge \neg y)=\neg x\vee \neg \neg y$,

\item[\rm (B8)]  $\nd (x\vee \nd y)=\nd x\wedge \nd \nd y$,

\item[\rm (B9)]  $(x\vee y)\wedge \nd (x\vee y)\leq x\vee \neg x$,

\item[\rm (B10)]  $x\wedge \nd x\wedge y\wedge \nd y\leq \nd (x\vee y)$,
\end{itemize}

\noindent where $a\leq b$ if and only if $a\wedge b=a$ or $a\vee b=b$. Moreover, $\neg x$ and $\nd x$ denote $\simi  \triangledown x$, and $\triangledown \simi  x$ respectively.
\end{teo}

The verification of the neccesary condition does not offer any special
difficulty; therefore we omit the proof. For the sufficient condition we
need the following lemmas and corollaries:

\

\begin{lem}
If $\langle A,\wedge ,\vee ,\neg ,\nd ,0,1\rangle$ is an algebra of type 
$(2,2,1,1,0,0)$ which verifies the properties {\rm B1, $\cdots$, B10} of theorem \ref{2.1} then it holds:

\begin{itemize}
\item[\rm (B11)]  $\nd 0=1$,

\item[\rm (B12)]  $\neg 1=0$,

\item[\rm (B13)]  $\neg x\leq \nd x$,

\item[\rm (B14)]  $\neg 0=1$,

\item[\rm (B15)]  $\nd 1=0$,

\item[\rm (B16)]  $\nd x\wedge \nd \nd x=0$,

\item[\rm (B17)]  $\neg x\vee \neg \neg x=1$,

\item[\rm (B18)]  $\neg \neg x=\nd \neg x$,

\item[\rm (B19)]  $\nd \nd x=\neg \nd x$,

\item[\rm (B20)]  $\neg x\wedge \nd \nd x=0$,

\item[\rm (B21)]  $x\leq \neg \neg x$,

\item[\rm (B22)]  $\nd \nd x\leq x$,

\item[\rm (B23)]  $\neg \neg \neg x=\neg x$,

\item[\rm (B24)]  $\nd \nd \nd x=\nd x$,

\item[\rm (B25)]  $\neg \nd x\leq x$,

\item[\rm (B26)]  $\nd \nd \neg x=\neg x$,

\item[\rm (B27)]  $\nd \nd \nd \neg x=\neg \neg x$,

\item[\rm (B28)]  $\nd ((x\vee \neg x)\wedge \nd x)=\neg \neg x$,

\item[\rm (B29)]  $\neg \neg \nd x=\nd x$,

\item[\rm (B30)]  $\neg \neg \neg \nd x=\nd \nd x$,

\item[\rm (B31)]  $\neg ((x\wedge \nd x)\vee \neg x)=\nd \nd x$,

\item[\rm (B32)]  $\nd \neg \neg x=\neg x$.
\end{itemize}
\end{lem}

\

\begin{proof} We only check B18, B22, B28 and B31.

\begin{itemize}
\item[\rm (B18)]  Then $\nd \neg x\leq \neg \neg x$ and by B13 $\neg \neg
x\leq \nd \neg x$.

\item[\rm (B22)]  $x=x\vee \nd 1$, \hfill [B15]

\hspace{2.5mm} $=x\vee \nd (x\vee \nd x)$, \hfill [B2]

\hspace{2.5mm} $=x\vee (\nd x\wedge \nd \nd x)$, \hfill [B8]

\hspace{2.5mm} $=x\vee \nd \nd x.$ \hfill [B2]

\item[\rm (B28)]  $\nd ((x\vee \neg x)\wedge \nd x)=\nd (x\vee \neg
x)\vee \nd \nd x)$, \hfill [B5]

\hspace{30mm} $=\nd (x\vee \nd \nd \neg x)\vee \nd \nd x$, \hfill [B26]

\hspace{30mm} $=(\nd x\wedge \neg \neg x)\vee \nd \nd x$, \hfill[B8,B27]

\hspace{30mm} $=\neg \neg x$. \hfill [B2,B21,B22]

\item[\rm (B31)]  $\neg ((x\wedge \nd x)\vee \neg x)=\neg (x\wedge \nd
x)\wedge \neg \neg x)$, \hfill [B6]

\hspace{31mm} $=\neg (x\wedge \neg \neg \nd x)\wedge \neg \neg x$, \hfill [B29]

\hspace{31mm} $=(\neg x\vee \nd \nd x)\wedge \neg \neg x$, \hfill [B7,B30]

\hspace{31mm} $=\nd \nd x$. \hfill [B1,B21,B22]
\end{itemize}
\end{proof}

\begin{cor} {\rm (Axiom A4)} $\sim \simi  x=x.$
\end{cor}

\begin{proof} First, we observe that from B13 and D2 we obtain

\begin{itemize}
\item[\rm (D3)]  $\simi  x=(x\wedge \nd x)\vee \neg x$.
\end{itemize}

Then

$\sim \simi  x=(((x\wedge \nd x)\vee \neg x)\wedge \nd ((x\wedge
\nd x)\vee \neg x))\vee \neg ((x \wedge \nd x)\vee \neg x)$, \hfill [D3]

\hspace{8.5mm} $=(((x\wedge \nd x)\vee \neg x)\vee \neg \neg x)\vee \nd
\nd x$, \hfill [B28,B31,D3,D2]

\hspace{8.5mm} $=((x\wedge \nd x\wedge \neg \neg x)\vee (\neg x\wedge \neg
\neg x))\vee \nd \nd x$,

\hspace{8.5mm} $=(x\wedge \nd x\wedge \neg \neg x)\vee \nd \nd x$, \hfill [B1]

\hspace{8.5mm} $=x$. \hfill [B22,B2,B21,B22]

\end{proof}

\begin{cor}{\rm (Axiom A6)} $\simi  x\vee \triangledown x=1$.
\end{cor}

\begin{proof}

$\simi  x\vee \triangledown x=((x\vee \neg x)\wedge \nd x)\vee \neg \neg x$, \hfill [D2,D1]

\hspace{15.5mm} $=(x\wedge \nd x) \vee \neg x \vee \neg \neg x$, \hfill [D3]

\hspace{15.5mm} $=(x \vee \nd x) \vee 1=1$. \hfill [B17]

\end{proof}

\begin{cor}{\rm (Axiom A7)} $x\wedge \simi  x=\simi  x\wedge \triangledown x$.
\end{cor}

\begin{proof}

$\simi  x\wedge \triangledown x=((x\wedge \nd x)\vee \neg x)\wedge \neg \neg x$, \hfill [D3,D1]

\hspace{17mm}$=(x\wedge \nd x \wedge \neg \neg x) \vee (\neg x\wedge \neg \neg x)$,

\hspace{17mm}$=x\wedge \nd x=(x\wedge \nd x)\vee 0$, \hfill[B21,B1]

\hspace{17mm}$=(x\wedge \nd x) \vee (x \wedge \neg x)$, \hfill[B1]

\hspace{17mm}$=((x\wedge \nd x) \vee x) \wedge ((x \wedge \nd x) \vee \neg x)$,

\hspace{17mm}$=x \wedge \simi x$.\hfill[D3]

\end{proof}

\begin{lem}
The following properties hold:

\begin{itemize}
\item[\rm (B33)]  if $x\leq y$ then $\neg y\leq \neg x$ and $\nd y\leq \nd x
$,

\item[\rm (B34)]  $\simi  \nd x=\nd \nd x$,

\item[\rm (B35)]  $\simi  (\neg x\wedge \nd y)=\neg \neg x\vee \nd \nd y$,

\item[\rm (B36)]  $\nd \nd (y\vee \neg \neg x)=\nd \nd y\vee \neg
\neg x$,

\item[\rm (B37)]  $\nd \nd (x\vee y)=\neg \neg x\vee \nd \nd y$,

\item[\rm (B38)]  if $x\leq y$ then $\simi  y\leq \simi  x$,

\item[\rm (B39)]  $\neg x\wedge \nd y\leq \nd \left( x\vee y\right) ,$

\item[\rm (B40)]  $\neg x\wedge \simi  y\leq \left( x\vee y\right) \vee \neg \left(
x\vee y\right) ,$

\item[\rm (B41)]  $x\wedge \nd x\wedge \simi  y\leq \nd \left( x\vee y\right)
,$

\item[\rm (B42)]  $\neg x\wedge \simi  y\leq \nd \left( x\vee y\right) .$
\end{itemize}
\end{lem}

\begin{proof}

We check only B34, B35, B36, B38, B39, B40 and B41.

\begin{itemize}
\item[\rm (B34)]  $\simi  \nd x=\nd \nd x \wedge (\nd x \vee \neg \nd x)
$, \hfill [D2]

\hspace{9.5mm}$=\nd \nd x \wedge (\nd x\vee \nd \nd x)$, \hfill[B19]

\hspace{9.5mm}$=\nd \nd x$, \hfill[B2]

\item[\rm (B35)]  (1) $\simi   \left( \neg x\wedge \nd y\right) =\nd
\left( \neg x\wedge \nd y\right) \wedge \left( \left( \neg x\wedge \nd
y\right) \vee \neg \left( \neg x\wedge \nd y\right) \right). $\hfill[D2]

On the other hand (2) $\nd \left( \neg x\wedge \nd y\right) =\nd \neg x\lor \nd \nd y$, \hfill[B5]

\hspace{28.5mm}$=\neg \neg x\lor \nd \nd y$, \hfill [B18]

and

(3) $\neg \left( \neg x\wedge \nd y\right) =\neg \nd y\vee \neg \neg x$, \hfill[B7]

\hspace{28mm}$=\neg \neg x\lor \nd \nd y$, \hfill[B19]

Then B35 follows from (1), (2) and (3).

\item[\rm (B36)]  $\nd \nd y\vee \neg \neg x=\nd \nd y\vee \nd \nd \neg \neg x$, \hfill[B26]

\hspace{20.5mm}$=\nd \left( \nd y\land \nd \neg \neg x\right) $, \hfill[B5]

\hspace{20.5mm}$=\nd \left( \nd y\land \neg x\right) $, \hfill[B32]

\hspace{20.5mm}$=\nd \left( \nd y\land \nd \nd \neg y\right) $, \hfill[B26]

\hspace{20.5mm}$=\nd \nd \left( y\vee \nd \neg x\right) $, \hfill[B8]

\hspace{20.5mm}$=\nd \nd \left( y\vee \neg \neg x\right) $, \hfill[B18]

\item[\rm (B38)]  Let $x,y$ be such that

\begin{itemize}
\item[\rm (1)]  $x\leq y$.

Then

\item[\rm (2)]  $\simi  y\vee \simi  x=\left( y\wedge \nd y\right) \vee \neg
y\vee \neg x\vee \left( x\wedge \nd x\right) $, \hfill[D3]

\hspace{16mm}$=\left( y\wedge \nd y\right) \vee \neg x\vee \left(x\wedge \nd x\right) $, \hfill[(1),B34]

\hspace{16mm}$=\neg x\vee \left( \left( y\vee x\right) \wedge \left(
y\vee \nd x\right) \wedge \left( \nd y\vee x\right) \wedge \left(
\nd y\vee \nd x\right) \right) $,

\hspace{16mm}$=\neg x\vee \left( y\wedge \left( y\vee \nd x\right)
\wedge \left( \nd y\vee x\right) \wedge \nd x\right)$. \hfill((1),B34)

Furthermore

\item[\rm (3)]  $1=x\vee \nd x$, \hfill[B2]

\hspace{4mm}$\leq y\vee \nd x$, \hfill[(1)]

Then

\item[\rm (4)]  $\simi  y\vee \simi  x=\neg x\vee \left( y\wedge \left( \nd y\vee
x\right) \wedge \nd x\right) $, \hfill[(2),(3)]

\hspace{16mm}$  =\neg x\vee \left( \nd x\wedge \left( \left( y\wedge
\nd y\right) \vee \left( y\lor x\right) \right) \right) $,

\hspace{16mm}$  =\nd x\wedge \left( \neg x\vee \left( y\wedge y\right)
\vee x\right) $, \hfill[(1),B13]

\item[\rm (5)]  $y\wedge \nd y\leq x\vee \neg  x$. \hfill[(1),B9]
\end{itemize}

\begin{itemize}
\item[] Then

$\simi  y\vee \simi  x=\nd x\wedge \left( \neg x\vee x\right) $, \hfill[(4),(5)]

\hspace{16mm}$=\simi  x$, \hfill[D2]

\item[\rm (B39)]  From B34 and B35 we have

$\simi  \nd \left( x\vee y\right) =\nd \nd \left( x\vee y\right) $,

$\simi  \left( \neg x\land y\right) =\neg \neg x\vee \nd \nd y$,

and by B37 it results

(1) $\simi  \nd \left( x\vee y\right) \leq \simi  \left( \neg x\vee \nd
y\right) $

From (1), B38 and corollary 2.3. $\neg x \wedge \nd y \leq \nd(x \wedge y)$. 

\item[\rm (B40)]  $\neg x\wedge \simi  y\wedge \left( \left( x\vee y\right) \vee
\neg \left( x\vee y\right) \right) =\simi  y\wedge \left( \left( \neg x\wedge \left( x\vee y\right) \right) \vee \left(\neg x\wedge \neg \left( x\vee y\right) \right) \right) $,

\hspace{55mm}$=\simi  y\wedge \left( \left( \neg x\wedge
y\right) \vee \left( \neg x\wedge \neg y\right) \right) $, \hfill[B1,B33]

\hspace{55mm}$=\simi  y\wedge \neg x\wedge \left( y\vee \neg y\right) $,

\hspace{55mm}$=\neg x \wedge \simi y$. \hfill [D3]

\item[\rm (B41)] (1) $x\wedge \nd x\wedge \simi  y=x\wedge \nd x\wedge \nd y\wedge \left( y\vee \neg y\right) $,\hfill[D3]
	
\hspace{29mm}$=\left( x\wedge \nd x\wedge \nd y\wedge y\right) \vee \left( x\wedge \nd x\wedge \nd y\wedge \neg y\right)$.

On the other hand 

(2) $x\wedge \nd x\wedge y\wedge \nd y\leq \nd \left( x\vee y\right)$,\hfill[B10]
	
(3)\hspace{15mm}$\nd x\wedge \neg y\leq \nd \left( x\vee y\right)$. \hfill[B39]

Then

$x\wedge \nd x\wedge \simi  y\leq \nd \left( x\vee y\right) \vee \left( x\wedge \nd \left( x\vee y\right) \right) =\nd \left( x\vee y\right) $. \hfill[(1)(2)(3)]
\end{itemize}
\end{itemize}
\end{proof}

\begin{cor} \rm (Axiom A5)
$ \simi  (x \vee y) = \simi  x \wedge \simi  y$
\end{cor}
\begin{proof}

We have

$ \simi  (x \vee y) = \simi  x \wedge \simi  y$, \hfill [B38]

On the other hand

\begin{enumerate}
\item[\rm (1)]  $\simi  x\wedge \simi  y=\left( x\wedge \nd x\wedge \simi  y\right) \vee \left( \neg x\wedge \simi  y\right) $, \hfill[D2]

\item[\rm (2)]  $ x\wedge \nd x\wedge \simi  y \leq \left(x\vee y\right) \vee \neg \left(  x \vee  y\right) $,

\item[\rm (3)]  $ x\wedge \nd x\wedge \simi  y \leq \simi  \left( x\vee y\right), $\hfill[(3),(B41),(D2)]

\item[\rm (4)]  $\neg x\wedge \simi  y\leq \simi  \left( x\vee y\right), $\hfill[(B40),(B42),(D2)]

\item[\rm (5)]  $\simi  x\wedge \simi  y\leq \simi  \left( x\vee y\right) $. \hfill[(1),(2),(4)]
\end{enumerate}

Finally, taking into account that $\left( A,\wedge ,\vee ,0,1\right)$ is a bounded distributive \linebreak lattice with least element 0, greatest element 1, the sufficient condition of theorem 2.1. follows from corollaries 2.3, 2.7, 2.4 and 2.5.  
\end{proof}

\section{Other characterizations}

The following characterization of 4--valued modal algebras is easier than that given in theorem 2.1. 
\begin{teo}
Let $\left( A,\wedge ,\vee ,\simi ,\neg ,1\right) $ be an algebra of type {\rm(2,2,1,1,0)} where $\left( A,\wedge ,\vee ,\sim,1\right) $ is a De
Morgan algebra with last element {\rm 1} and first element $0=\simi  1$. If  $\triangledown$ is an unary
operation defined on $A$ by means of the formula $\triangledown x=\simi  \neg x$. Then $A$ is a {\rm 4}--valued modal algebra if and only if it verifies:

\begin{enumerate}
\item[\rm (T1)]  $x\wedge \neg x=0$.

\item[\rm (T2)]  $x\vee \neg x=x\vee \simi  x$.
\end{enumerate}

Furthermore $\neg x = \simi  \triangledown x$.
\end{teo}

\begin{proof}

We check only sufficient condition 

\begin{enumerate}
\item[\rm (A6)]  $\simi  x\vee \triangledown x=\simi  x\vee \simi  \neg x=\simi  \left( x\wedge \neg x\right)=1 $. \hfill[T1]

\item[\rm (A7)]  $\simi  x\wedge \triangledown x=\simi  x\wedge \simi  \neg x$, \hfill[T2]

\hspace{17.5mm}$=\simi  \left( x\vee \neg x\right) $,

\hspace{17.5mm}$=\simi  \left( x\vee \simi  x\right) $,

\hspace{17.5mm}$=x\wedge \simi  x$.
\end{enumerate}
\end{proof}

\begin{rem}
In a {\rm 4}--valued modal algebra the operation considered in {\rm 2.1},
generally does not coincide with the pseudo-complement $^*$ as we can verify in the following example:

\setlength{\unitlength}{1cm}
\begin{minipage}{5cm}
\begin{picture}(2,2)
\put(3,0){\circle*{.15}}
\put(2,1){\circle*{.15}}
\put(4,1){\circle*{.15}}
\put(3,2){\circle*{.15}}
\put(3,0){\line(-1,1){1}}
\put(2,1){\line(1,1){1}}
\put(3,0){\line(1,1){1}}
\put(4,1){\line(-1,1){1}}
\put(1.6,1){$a$}
\put(2.9,-0.5){$0$}
\put(4.2,1){$b$}
\put(2.9,2.2){$1$}
\put(2.3,-1){\rm  figure 1}
\end{picture}
\end{minipage}
\hspace{2cm}
\begin{minipage}{5cm}
\begin{tabular}{c|c|c}
$x$ & $\simi  x$& $\triangledown x$ \\ \hline
$0$ & $1$ &  $0$\\
 $a$ & $a$ &  $1$\\
 $b$ & $b$ &  $1$\\
 $1$ & $0$ &  $1$ 
\end{tabular}

\begin{picture}(1,1)
\put(1,-0.2){\rm  table 1}
\end{picture}
\end{minipage}

\vspace{1cm} 

we have

\setlength{\unitlength}{1cm}
\hspace{5cm}
\begin{minipage}{5cm}
\begin{tabular}{c|c|c}
$x$ & $\neg x$& $x^*$ \\ \hline
$0$ & $1$ &  $1$\\
 $a$ & $0$ &  $b$\\
 $b$ & $0$ &  $a$\\
 $1$ & $0$ &  $0$ 
\end{tabular}

\begin{picture}(1,1)
\put(.5,0.5){\rm  table 2}
\end{picture}
\end{minipage}
\end{rem}

\noindent However all finite 4--valued modal algebra is a distributive lattice pseudo
\linebreak complemented. We do not know whether this situation holds in the \linebreak non-finite
case. This suggests that we consider a particular class of De  \linebreak Morgan algebras.

\begin{defi}
An algebra $\left( A,\wedge ,\vee ,\sim ,^*,1\right) $ of type {\rm (2,2,1,1,0)} is a modal De \linebreak Morgan $p$--algebra if the reduct $\left( A,\wedge ,\vee ,\sim ,1\right) $ is a De Morgan algebra with last element 1 and first element $0= \simi  1$, the reduct is a pseudo--complemented meet--lattice and the following condition is verified
\begin{center}

\rm H1) $x \vee \simi  x \leq x \vee x^*$

\end{center}
\end{defi}

\begin{exa}
The De Morgan algebra whose Hasse diagram is given in \linebreak {\rm figure 2} and the operations $\sim$ and $^*$ are defined in {\rm table 3}

\vspace{1cm}
\setlength{\unitlength}{1cm}
\begin{minipage}{5cm}
\begin{picture}(2,2)
\put(3,0){\circle*{.15}}
\put(3,1){\circle*{.15}}
\put(3,2){\circle*{.15}}
\put(3,3){\circle*{.15}}
\put(3,0){\line(0,1){3}}
\put(2.6,1){$a$}
\put(2.6,0){$0$}
\put(2.6,2){$b$}
\put(2.6,2.9){$1$}
\put(2.3,-1){\rm  figure 2}
\end{picture}
\end{minipage}
\hspace{2cm}
\begin{minipage}{5cm}
\begin{tabular}{c|c|c}
$x$ & $\simi  x$& $x^*$ \\ \hline
$0$ & $1$ &  $1$\\
 $a$ & $b$ &  $0$\\
 $b$ & $a$ &  $0$\\
 $1$ & $0$ &  $0$ 
\end{tabular}

\begin{picture}(1,1)
\put(1,-0.2){\rm  table 3}
\end{picture}
\end{minipage}

\vspace{.5cm}

is not a modal De Morgan $p$--algebra because $b=(a \vee \simi  a) \not\leq a \vee a^*=a$.
\end{exa}

\begin{teo}
If we define on a modal De Morgan $p$--algebra $\langle A,\wedge ,\vee ,\sim ,^*,1\rangle$ the operation $\neg$ by means of the formula $\neg x = x^* \wedge \simi  x$  then the algebra $\langle A,\wedge ,\vee ,\neg,1\rangle$ verifies the identities {\rm T1} and {\rm T2}.
\end{teo}
\begin{proof}
\begin{enumerate}
\item[\rm (T1)]  $x\wedge \neg x=x\wedge x^{*}\wedge \simi  x=0\wedge \simi  x=0$.

\item[\rm (T2)]  $x\vee \neg x=x\vee \left( x^{*}\wedge \simi  x\right) =\left( x\vee x^{*}\right) \wedge \left( x\vee \simi  x\right)$,

\hspace{13mm}$=\left( x\vee \simi  x\right) $. \hfill[H1]
\end{enumerate}
\end{proof}

\begin{rem}

By [4] we know that every finite modal {\rm 4}--valued algebra $A$ is direct product of copies of {\rm T2}, {\rm T3} and {\rm T4}, where {\rm T2=\{0,1\}} and 
{\rm T3=\{0,a,1\}} are modal De Morgan $p$--algebra we conclude that $A$ is also a modal De Morgan $p$--algebra.

We do not know whether this situation holds in the non-finite case.
\end{rem}

\noindent Instituto de Ciencias B\'asicas \\ Universidad Nacional de San Juan \\
Avda.Ignacio de la Roza 230 - Oeste \\ 5400 - San Juan Argentina

\end{document}